\documentclass[12pt,oneside]{amsart}

\usepackage{amssymb}

    \theoremstyle{plain}
      \newtheorem{theorem}{Theorem}[section]
	\newtheorem{proposition}{Proposition} 
      \newtheorem{lemma}[theorem]{Lemma}
      \newtheorem{corollary}[theorem]{Corollary}
      
      \theoremstyle{definition}
      \newtheorem{definition}[theorem]{Definition}
      
      \theoremstyle{remark}
      \newtheorem{remark}[theorem]{Remark}

\begin{document}


\title[algebraic closure operators]{The lattice of algebraic closure operators}

\author[M. L. H. Kilpack]{Martha Lee Hollist Kilpack}

  \email{Martha.Kilpack@oneonta.edu}
\address{Department of Mathematics, Computer Science, and Statistics \\
SUNY Oneonta\\New York 13820\\USA}


\thanks{The author thanks her understanding advisor Fernando Guzman.}




\begin{abstract}
On an infinite set some closure operators are finitary (algebraic) while others are not.  We can generalize this idea for a complete algebraic lattice letting the compact elements act as the finite sets.  With this in mind, we will consider the set of algebraic closure operators on such a lattice.  We will show this set forms a complete lattice that is also an algebraic lattice.



\end{abstract}

\maketitle


\section{Introduction}\label{S:intro}
Birkhoff \cite{top_B}, Ore \cite{ore}, and Ward \cite{ward}, were among the first to study the set of closure operators which are mappings from a given set to itself or a complete lattice to itself.  Birkhoff \cite{top_B} showed this set of mappings, acting on a set or a complete lattice, forms a complete lattice. 

More recently Ranzato \cite{ran-closures} has looked at the set of closure operators which are mappings from a given partially ordered set to itself and sufficient conditions on the partial order which make the set of closure operators a complete lattice.  

In the above mentioned cases and others, little if at all, is said about algebraic or finitary closure operators.  On an infinite set $S$, we can consider the set of closure operators, $c.o.(S)$, but also the subset of algebraic closure operators, $a.c.o(S)$. We will consider these algebraic closure operators and what this set looks like.  We would like to look more generally than sets and look at lattices. When looking at closure operators acting on lattices rather than acting on a  set we can generalize the idea of algebraic closure operators.  To do this will need elements of the lattice to act as finite subsets do in a set.  With this in mind we restrict our lattice to be an algebraic lattice.  The compact elements of the algebraic lattice  act as the finite subsets do in the lattice of subsets.  

Once we have this generalization we can consider the set of closure operators which are mappings of a given algebraic lattice $L$, $c.o.(L)$, and its subset of the algebraic closure operators of $L$, $a.c.o.(L)$.

\vspace{2mm}\noindent{\bf Proposition 3.7}. {\it 
 Let $L$ be an algebraic lattice. Then $a.c.o.(L)$ is a sublattice of $c.o.(L)$.  
}\vspace{2mm}

\vspace{2mm}\noindent{\bf Theorem 3.9}. {\it 
Let $L$ be an algebraic lattice. Then $a.c.o.(L)$ is a complete lattice. For a family $(\phi_i| i\in I)$ with $\phi_i \in a.c.o.(L)$ and $x \in L$,  
$$
\left( \bigwedge_{i\in I} \phi _i\right) (x) = \bigvee _{k \leq_c  x} \left( \bigwedge _{i \in I} \phi _i (k) \right) \hspace{.05in} \textnormal{and} \hspace{.05in}  \left( \bigvee_{i \in I}\phi_i \right) (x) = \bigvee_{j \in I^k,  k \geq 0} \phi_{j_1} \phi_{j_2} ...\phi_{j_k} (x). 
$$ 
}\vspace{2mm}

The final section will look at the complete lattice of $a.c.o.(L)$ find compact elements and then showing these elements generate $a.c.o.(L)$, making $a.c.o.(L)$ an algebraic lattice.  

When looking at the lattice of closure operators one might also consider closure systems.  The lattice of all closure systems from a given lattice $L$ is the dual of $c.o.(L)$ \cite{ore}. By duality we would find that the closure systems related to algebraic closure operators would be a sub-lattice of the the lattice of all the closure systems for a given lattice. Duality does not give us the same for showing the lattice of algebraic closure systems is an algebraic lattice. We will say a discussion of this for another time.

\section{Preliminaries}\label{prelims}

We will first be reminded of the following useful items from lattice theory and closure operator theory.    

A lattice is a non-empty partially ordered set $L$ such that for all 
$a$ and $b$ in $L$ both $a\vee b:= \sup\{a,b\}$ and $a \wedge b := \inf \{a,b\}$ exist.  
A partially ordered set $L$ is called a complete lattice when for each of the subsets $S$ of $L$ the $\sup\{S\}$ and the $\inf\{ S\}$ exists in $L$ \cite[Ch. I Sec. 4]{latticeB}.  An element $c$ in a lattice is compact if $ c \leq \bigvee_{i \in I} x_i$ implies $ c \leq \bigvee_{i \in F} x_i$ for some finite $F \subseteq I$.  We will let $C(L) := \{c \in L | c \textnormal{ is compact } \}$, and let $k \leq_c  x$ denote that $k \leq x$ and $k \in C(L)$.  A lattice for which every element is the join of compact elements is called compactly generated or algebraic \cite[Ch. VIII Sec 4$\&$ 5]{latticeB}.  
 
We will note that in an algebraic lattice to show $x \leq y$ it is enough to show for all $k \leq_c x$, we have $k \leq y$. 

\begin{definition}\label{closure_lat}\cite[Ch.V Sec. 1]{latticeB} Given a lattice $L$, a mapping $\phi:L \rightarrow L$ is called a \textsl{closure operator} on $L$ if for $a, b \in L$, it satisfies: \\
C1: $a \leq \phi(a)$ \hfill  (extensive)\\
C2: $\phi(\phi(a)) = \phi(a)$ \hfill (idempotent)\\
C3: If $a \leq b$ then $\phi(a) \leq \phi(b)$ \hfill (isotone) \\
Let $c.o.(L) :=  \left\{ \phi | \phi \textnormal{ is a closure operator} \right\}$
\end{definition}

For mapping which act of the power set lattice there is the idea of finitary mappings.  For a set $S$ a mapping $\phi: \mathcal{P}(S) \rightarrow \mathcal{P}(S)$ is called finitary if for all $A \subseteq S$  $\phi(A) = \bigcup_{F \textnormal{ finite subset } A} \phi(F) $ \cite[Ch. VIII Sec. 4]{latticeB}.   

As we wish to look at all algebraic lattices and not just the power set lattices we will extend this definition for such closure operators.  

\begin{definition}\label{finitary} Let $L$ be an algebraic lattice.  An operator $\phi:L \rightarrow L$ is called finitary if for all $x \in L$:\\ 
F: $\displaystyle{ \phi(x) = \bigvee_{k \leq_c x} \phi(k)} $.  
\\ Let $a.c.o.(L) := \left\{ \phi \in c.o.(L) | \phi \textnormal{ is finitary}\right\}$ and be called set of algebraic closure operators.   
\end{definition}

\begin{remark}\label{geq_alg}
For any isotone operators $\displaystyle{ \phi(x) \geq \bigvee_{k \leq_c x} \phi(k)} $ is always true.  
\end{remark} 

\begin{theorem}\label{lattice_comp_L}\cite{top_B} Let $L$ be a complete lattice then $c.o.(L)$ is a complete lattice where $\phi_1 \leq \phi_2$ if and only if $\phi_1 (x) \leq_L \phi_2(x)$ for all $x \in L$ with the meets and the joins defined as follows: 
$$
\left( \bigwedge_{i \in I} \phi_i\right) (x) := \bigwedge_{i \in I} \left( \phi_i(x) \right) \ \ \forall \ x \in L  
$$
and 
$$
\left( \bigvee_{i \in I} \phi_i \right) (x) := \bigwedge \left\{ c \in L | c \geq x \textnormal{ and } \phi_i(c) = c \ \forall \ i \in I \right\}. 
$$
\end{theorem}

At this point let us mention we will be moving between a general lattice $L$ and the lattice of closure operators $c.o.(L)$ without distinguishing in which lattice we are taking the meet or join, as in the definition above.  $\left( \bigvee_{i \in I} \phi_i \right) (x)$ would be taking the join in $c.o.(L)$ and then taking the closure of $x$ under the new operator.  Where $ \bigvee_{i \in I} \left( \phi_i (x) \right)$ would be a join in $L$ where the elements joined, $\phi_i (x)$, are the closures of $x$ under the different closure operators.  

\section{Sublattice}\label{sub} 

In Theorem~\ref{lattice_comp_L}, we see that set of all closure operators which act on a lattice form a complete lattice.  In this section we will consider a subset of that lattice, the set of all algebraic closure operators.  Please note in a finite lattice the set of closure operators would be the same as the set of algebraic closure operators.  Thus the more interesting case with when a lattice is infinite.  We restrict our infinite lattices to algebraic lattice so that we can more easily extend the definition for a finitary operator, or more specifically an algebraic (finitary) closure operator, Definition~\ref{finitary}.  

We will first consider what would happen if we take the meet of two algebraic closure operators.

\begin{proposition}\label{meetL}
Let $L$ be an algebraic lattice and let $\phi_1, \ \phi_2 \in a.c.o.(L)$. Then $\phi_1 \wedge \phi_2 \in a.c.o.(L)$. 
\end{proposition}

\begin{proof} 
Let $L$ be an algebraic lattice and let $\phi_1, \ \phi_2 \in a.c.o.(L)$.  Let $x \in L$ and let $y = (\phi_1 \wedge \phi_2)(x) = \phi_1(x) \wedge \phi_2(x) = \left(\bigvee_{k \leq_c x} \phi_1 (k)\right) \wedge \left(\bigvee_{l \leq_c x} \phi_2(l)\right)$. 
Since $L$ is algebraic,  $ y=\bigvee_{a \leq_c y} a$ and for all $a \leq_c y$, 
$ a \leq (\bigvee_{k \leq_c x} \phi_1 (k))$ and $a \leq (\bigvee_{l \leq_c x} \phi_2(l))$.    
By the compactness of $a$ these covers can be reduced to finite covers. Let $M_a$ and $N_a$ be such covers.  Then 
$$
a \leq \bigvee_{m\in M_a} \phi_1 (m), \hspace{.85in}   a \leq \bigvee_{n\in N_a} \phi_2(n).
$$ 
Letting $ a^* =\left( \bigvee_{j \in M_a \cup N_a} j \right)$ making $a^* \leq_c x$.  We now use this along with properties of closure operators to find 
$$
a \leq \phi_1(a^*) \wedge \phi_2(a^*) = (\phi_1 \wedge \phi_2)(a^*) \leq \bigvee_{k \leq_c x} (\phi_1 \wedge \phi_2)(k).
$$

We then see  
$$
y = \bigvee_{a \leq_c y} a  \leq  \bigvee_{k \leq_c x} (\phi_1 \wedge \phi_2)(k), \hspace{.35in} 
\left(\phi_1 \wedge \phi_2\right)(x) \leq \bigvee_{k \leq_c x } \left( \left(\phi_1 \wedge \phi_2\right (k) \right). 
$$  
With this and Remark~\ref{geq_alg} we have $\phi_1 \wedge \phi_2 (x) = \bigvee_{k \leq_c x } \left( \left(\phi_1 \wedge \phi_2\right) (k) \right)$ which by definition of algebraic closure operators makes $\phi_1 \wedge \phi_2 \in a.c.o.(L)$. 
\end{proof}

We could ask if we are also closed under arbitrary meets.  Below we have an example where the arbitrary meet of algebraic closure operators is not an algebraic closure operator.  

Let $L$ be an algebraic lattice with at least one non-compact element with is not the greatest element.  We will let 1 denote the greatest element of $L$.  Consider the following family closure operators which act on $L$.   Let $a\in L$. 
\begin{equation}\label{finite_phi}  
 \phi_a(x) := \left\{ \begin{array}{cl} \nonumber
 a, & \textrm{if } x \, \leq a\\ 
 1, & otherwise. \\
 
 \end{array}\right. 
\end{equation} 
This forms a family of algebraic closure operators. 
Now let us look at the arbitrary meet of $\phi_k$ for $k \in C(L)$.   
\begin{equation}\label{meet_finite_phi}
  \left( \bigwedge_{k \in C(L)}  \phi_k  \right) (x) = \left\{ \begin{array}{cl} \nonumber
 x, & \textrm{if } x \in C(L) \\ 
 1, & otherwise. \\
 
 \end{array}\right.
\end{equation} 

This meet is not algebraic. Thus,  $a.c.o.(L)$ is not closed under arbitrary meets.

We turn our attention to joins.  The next two lemmas come in useful when looking at the join of algebraic closure operators and the proofs for these lemmas are left to the reader.

\begin{lemma}\label{finite_c3}
The Property C3 of Definition~\ref{closure_lat} holds for finitary operators.  
\end{lemma}

\begin{lemma}\label{a_lee_mu}
Let $L$ be an algebraic lattice and let $\phi:L \rightarrow L$ be a mapping with Property C3. Then $\phi$ is finitary if and only if 
for $x \in L$ and $a \leq_c \phi(x)$ implies $a \leq \phi(l)$ for some $l \leq_c x$.   
\end{lemma}

With the use of these lemmas we will build a way to find the arbitrary join of  algebraic closure operators.  We will start by looking at the composite of many finitary operators.

 \begin{corollary}\label{corollary_finitary2_L}
Let $\phi_i$ be finitary operators for $1 \leq i \leq n$ where $n \in \mathbb{N}$.  Then $\phi_n \phi_{n-1}...\phi_2\phi_1$ is finitary.   
\end{corollary}

\begin{proof} Let $\phi_i$ be finitary operators for $1 \leq i \leq n$ where $n \in \mathbb{N}$. If $y \leq_L z$ then $\phi _2\phi_1(y) \leq_L \phi _2\phi_1(z)$ by reiterating Property C3.  This would make 
$ \bigvee_{k \leq_c x} (\phi_2\phi _1(k)) \leq_L \phi _2\phi_1(x)$. 
Take an $x\in L$. Let $\displaystyle a \leq_c \phi_2\phi_1(x) $.  
From Lemma~\ref{a_lee_mu} and the fact that $\phi_1$ is finitary closure operator, we have that \\
$ a~\leq~\phi_2~(l)$ for some $l \leq_c \phi_1(x)$.  We do the same for $l$ and find $l \leq \phi_1 (m)$ where
$m \leq_c x$.  We then have $ a \leq \phi_2(l) \leq \phi_2\phi_1(m) \leq \bigvee_{k \leq_c x} \phi_2\phi_1 (k). $. Thus, $\phi_2\phi_1$ is finitary by Lemma~\ref{a_lee_mu}. 

Using induction, assuming that $\phi _{n-1}\phi _{n-2}...\phi_1$
is finitary then we have that $\phi _n\phi _{n-1}...\phi_1$ is finitary.
\end{proof}

\begin{lemma}\label{finitary_arb_mu_L}
Let $L$ be an algebraic lattice. For a family $(\phi_i | i \in I)$ where $\phi_i \in a.c.o.(L)$, let
$$
\mu(x) := \bigvee_{j \in I^k,  k \geq 0} \phi_{j_1} \phi_{j_2} ...\phi_{j_k} (x). 
$$
Then \\
(1) for $x\in L$ and $a \in C(L)$, $a \leq \mu(x)$ if and only if $a \leq \phi_{j_1} \phi_{j_2} ...\phi_{j_k} (x)$ for some $k \geq 0$ and $j \in I^k$;  \\
(2) $\mu$ is finitary.
\end{lemma}
\begin{proof}

Let $x \in L$ and $a \in C(L)$. For $a \leq \phi_{j_1} \phi_{j_2} ...\phi_{j_k} (x)$ by definition of joins this would make $a \leq \mu (x) $.   For $a \leq_c \mu(x) = \bigvee_{j \in I^k,  k \geq 0} \phi_{j_1} \phi_{j_2} ...\phi_{j_k} (x)$. Because $a$ is compact, we have that $a$ is less then the join of a finite subset of $\left\{  \phi_{j_1} \phi_{j_2} ...\phi_{j_k} (x) | j \in I^k  and k \geq 0 \right\}$.  This is a finite set thus there is a $\phi_{i_1} \phi_{i_2} ...\phi_{i_k} (x)$ for some $i \in I^k$, $k\geq0$ where $a \leq \phi_{i_1} \phi_{i_2} ...\phi_{i_k} (x)$.

We know from Corollary \ref{corollary_finitary2_L} that $\phi_{j_1} \phi_{j_2} ...\phi_{j_k}$ is finitary for any $j \in I^k$ and $k \geq 0$.   Let $x \in L$ and $a \in C(L)$ with $a\leq \mu(x)$. 
Since $\phi_{j_1} \phi_{j_2} ...\phi_{j_k}$ is finitary there exists $l \leq_c x$  such that $a \leq \phi_{j_1} \phi_{j_2} ...\phi_{j_k}(l) \leq \mu (l)$.
  Since $l$ is compact,
$ a \leq \bigvee_{d \leq_c x} \mu(d)$.   We can do this for any $a \leq_c \mu(x)$. From this we have
$$
\mu(x) = \left( \bigvee_{a \leq_c \mu(x)}a \right) \leq \bigvee_{d \leq_c x}\mu(d).
$$
Since $\mu$ is isotone we have the other inclusion from Remark~\ref{geq_alg}. 
Thus we have that $\mu$ is finitary.    
\end{proof}

\begin{proposition}\label{join_arb_alg_L}
Let $L$ be an algebraic lattice. For a family $( \phi_i | i\in I)$ where $\phi_i \in a.c.o.(L)$, the join in $c.o.(L)$ is 
$$
\left( \bigvee_{i \in I}  \phi_i \right) (x)=\bigvee_{j \in I^k,  k \geq 0} \phi_{j_1} \phi_{j_2} ...\phi_{j_k} (x)  
$$
and therefore it is also the join in $a.c.o.(L)$.  
\end{proposition}
\begin{proof}
The right hand side is what we have called $\mu(x)$; let us continue to do so.  Since $\phi_{i} \leq \left( \bigvee_{i \in I}  \phi_i \right)$, we have $\mu(x) \leq \left( \bigvee_{i \in I} \phi_i\right)(x)$ for all $x \in L$.  
We will assume for the rest of the proof that we are choosing $x, y \in L$.  
\\
C1: $x \leq \phi_i (x) \leq \mu(x)$ for any $i \in I$. Thus, $\mu$ has the extensive property for closure operators.   
\\
C3: We have already shown that $\mu$ is finitary and from Lemma~\ref{finite_c3}, Property C3 holds for $\mu$. 
\\
C2: Having property C1 for $\mu$ we know $\mu(x) \leq \mu(\mu(x))$. Let $a \leq_c \mu(\mu(x))$, then by Lemma~\ref{a_lee_mu}
there exists $\displaystyle l \leq_c \mu(x)$ such that $\displaystyle a \leq \mu(l)$. Now we know $l$ compact, so by Lemma~\ref{finitary_arb_mu_L}, $l$ is less then some finite sub-cover, which in this case makes $l \leq \phi_{j_1} \phi_{j_2} ...\phi_{j_k}(x)$ for some $j \in I$ and $k \geq 0$.  This in turn gives us 
$$
a \leq \mu(l)\leq \mu \left( \phi_{j_1} \phi_{j_2} ...\phi_{j_k} (x) \right) = \mu(x). 
$$
Since for any compact elements where $a \leq_c \mu\mu(x)$ we have $a \leq_c \mu(x)$, thus $\mu\mu(x) \leq \mu(x)$.  Thus Property C3 holds for $\mu$.  

Having all three properties of a closure operator $\mu \in a.c.o.(L)$.  We have for each $i \in I$, $\phi_i \leq \mu$ and $\mu (x) \leq \left( \bigvee_{i \in I}  \phi_i \right)(x)$,  making \\ 
$ \mu =  \left( \bigvee_{i \in I}  \phi_i \right) $.  
\end{proof}

\begin{proposition}
 Let $L$ be an algebraic lattice. Then $a.c.o.(L)$ is a sublattice of $c.o.(L)$.  
\end{proposition}

We see the set $a.c.o.(L)$ is a lattice.  We see from Proposition~\ref{join_arb_alg_L} we have arbitrary joins. We shall see from the following proposition that although it is not a complete sublattice of $c.o.(L)$, $a.c.o.(L)$ is a complete lattice.

\begin{proposition} \label{def_arb_meet_L}
Let $L$ be an algebraic lattice. For a family $(\phi_i | i \in I)$, where $\phi _i \in a.c.o.(L)$ and $x \in L$, 
$$ 
\left( \bigwedge_{i\in I}^a \phi _i\right) (x) = \bigvee _{k \leq_c x}^L \left( \bigwedge _{i \in I} \phi _i (k) \right).  
$$
\end{proposition}

\begin{proof}
For ease of notation let $\tau (x) = \bigvee _{k \leq_c x} \left( \bigwedge _{i \in I} \phi _i (k) \right)$ for all $x \in L$.  We will show $\tau$ is an algebraic closure operator, that is, that properties F1, C1, C2, and C3 hold for $\tau$.   \\
F1: Consider what $\tau(k)$ would be if $k \in C(L)$: 
$$
\tau(k) = \bigvee _{l \leq_c k}^L \left( \bigwedge_{i \in I} \phi _i (l)\right) = \bigwedge_{i \in I} \phi _i (k)  \hspace{.15in} \textnormal{ because } k \leq_c k, 
$$  
Let $x \in L$ then  $\tau(x) = \bigvee _{k \leq_c x} \left( \bigwedge_{i \in I} \phi _i (k) \right) = \bigvee _{k \leq_c x} \left( \tau(k) \right),$ thus $\tau$ is finitary.  \\ 
C1: Let $x \in L$. Since $\tau$ is finitary and $L$ is algebraic for $k \leq_c x$ then \\
$ k\leq \bigwedge_{i\in I} \phi_i (k) = \tau(k) \leq \tau(x)$. Thus we have $x = \bigvee_{k \leq_c x} k  \leq \tau(x)$.\\  
C3: Since we already know $\tau$ is finitary we get Property C3 by Lemma~\ref{finite_c3}. 
\\
C2: Let $k \leq_c \tau(\tau(x))$.  By Lemma~\ref{a_lee_mu} there is an $a_k \leq_c \tau(x)$ such that $\displaystyle k \leq \tau(a_k)$.  
Similarly, there is a $d_k \leq_c x $ such that $ a_k \leq \tau(d_k) = \bigwedge_{i\in I} \phi_i(d_k)$.  
Which is to say that $a_k \leq \phi_i(d_k)$ for all $i \in I$.   
We once again employ Property~C3 on the $\phi_i$'s and also Property~C2 to find $\phi_i(a_k) \leq \phi_i(\phi_i(d_k)) = \phi_i(d_k)$ for all $i\in I$. For any $k \leq_c \tau\tau (x)$. 
$$
k \leq \tau(a_k) = \bigwedge_{i\in I} \phi _i (a_k) \leq \bigwedge_{i\in I} \phi_i (d_k) = \tau(d_k) \leq \tau (x). 
$$
Thus, $\tau \tau (x)  \leq \tau(x)$ making $\tau\tau(x) = \tau(x)$.
 
Now to show that $\tau$ is the greatest lower bound for the $\phi_i$'s, let $\rho \in a.c.o.(L)$ such that $\rho$ is a lower bound for each of the $\phi_i$'s. 
Then $\rho (x) \leq \phi _i(x)$ for all $i\in I$. This would also make $ \rho ( a) \leq \phi _i(a)$ for all $i \in I$ and $a \leq_c x$. $ \rho(a) \leq \bigwedge_{i\in I} \phi_i(a) = \tau (a)$. We know that $\rho$ is finitary hence 
$$
\rho(x)= \bigvee_{a \leq_c x} \rho(a) \leq \bigvee_{a \leq_c x} \tau (a) = \tau(x). 
$$
This makes $\tau$ the greatest lower bound.
\end{proof}

\begin{theorem}
Let $L$ be an algebraic lattice. Then $a.c.o.(L)$ is a complete lattice. For a family $(\phi_i| i\in I)$ with $\phi_i \in a.c.o.(L)$ and $x \in L$,  
$$
\left( \bigwedge_{i\in I} \phi _i\right) (x) = \bigvee _{k \leq_c  x} \left( \bigwedge _{i \in I} \phi _i (k) \right) \hspace{.05in} \textnormal{and} \hspace{.05in}  \left( \bigvee_{i \in I}\phi_i \right) (x) = \bigvee_{j \in I^k,  k \geq 0} \phi_{j_1} \phi_{j_2} ...\phi_{j_k} (x). 
$$  
\end{theorem}


\section{Properties of the lattice of algebraic closure operators}\label{proposition_acoL}

In this section we will look at properties of the lattice $a.c.o.(L)$, find a set of elements that are compact and show that the elements of this set build the whole lattice.  To do this we will use the make use of the following from Birkhoff \cite[Ch. VIII Sec. 4 $\&$ 5] {latticeB}.  

For a directed set, a poset such that any two elements have an upper bound in the set, and a compact element $k$ in an algebraic lattice,  $k \leq \bigvee D$ if and only if $k \leq d$ for some $d \in D$.      

A complete lattice $L$ is said to be meet continuous when for any directed set $D \subseteq L$, and for any $a \in L$ we have the following, $a \wedge \left( \bigvee_{d\in D} d \right) = \bigvee_{d \in D} (a \wedge d)$.  Any complete algebraic lattice is meet-continuous. 

An element $a$ in a lattice $L$ is join-inaccessible when for a directed set $D \subseteq L$, $a = \bigvee_{d \in D} d$ implies $a = d$ for some $d \in D$.  In a complete meet-continuous lattice $L$, an element $c$ is compact if and only if $c$ is join-inaccessible.  

We would like to apply this last piece about the relationship of compact elements and join-inaccessible elements to help us distinguish some compact elements of $a.c.o.(L)$. To do this we will first prove $a.c.o.(L)$ is a meet continuous lattice, thus we will begin by looking at how directed sets behave in $a.c.o.(L)$.  
 
\begin{lemma}\label{dir_join_L}
Let $L$ be a complete algebraic lattice, let the family $( \sigma_\psi | \psi \in \Psi)$ be a directed set in 
$a.c.o.(L)$, then  $ \left( \bigvee_{\psi \in \Psi} \sigma_\psi \right)(x) = \bigvee_{\psi \in \Psi} \left( \sigma_\psi (x) \right)$.     
\end{lemma}

\begin{proof}
Consider $\sigma_n \sigma_{n-1}...\sigma_1$ where $\sigma_i \in( \sigma_\psi| d \in \Psi)$ for all $1 \leq i \leq n$, let $\sigma_{d}$ be the upper bound in the directed set. 
Making \\ 
$\sigma_n \sigma_{n-1}...\sigma_1(x)~\leq~(\sigma_{d})^n(x)~=~\sigma_{d}(x)$ 
for all $x~\in~L$.  Thus for any finite compose in the directed set we find 
$\sigma_n \sigma_{n-1}...\sigma_1(x)~\leq~\bigvee_{\psi \in \Psi}\left(\sigma_\psi (x)\right)$. 
This puts 
$$
\left( \bigvee_{\psi \in \Psi } \sigma_\psi \right) (x) =\bigvee_{ j \in \Psi^k, k\geq 1} \left( \sigma_{j_1} ... \sigma{j_k} (x) \right) \leq  \bigvee_{\psi \in \Psi} \left( \sigma_d (x) \right).
$$ 
Thus we have $\bigvee_{\psi \in \Phi} \left( \sigma_\psi (x) \right) = \left( \bigvee_{\psi \in \Psi} \sigma_\psi \right) (x) \hspace{.15in} \forall x \in L$.   
\end{proof} 

\begin{lemma}\label{meet_cont_L}
Let $L$ be a complete algebraic lattice. Then $a.c.o.(L)$ is a meet continuous lattice.  
\end{lemma}

\begin{proof}
 Let $ ( \sigma_\psi| \psi \in \Psi)$ be directed set in $a.c.o.(L)$, $\phi \in a.c.o.(L)$, and $x \in L$.  
Now we consider the following:
$$
\left( \phi \wedge \left(\bigvee_{\psi \in \Psi} \sigma_\psi \right) \right) (x) = \phi (x) \wedge \left(\bigvee_{\psi\in \Psi} \sigma_\psi \right)(x)= \phi (x) \wedge \left( \bigvee_{\psi\in \Psi} \sigma_\psi (x) \right).
$$ 
Let $\textit{X} = \left\{ \sigma_\psi(x) | \psi \in \Psi \right\} \subseteq L$ and consider $\sigma_1 (x), \sigma_2(x) \in  \textit{X}$; this would put $\sigma_1, \sigma_2$ in $ (\sigma_\psi | \psi \in \Psi )$, the directed set.  
So there must be an element $\sigma_{3}\in  (\sigma_\psi | \psi \in \Psi )$ such that $\sigma_1, \sigma_2 \leq \sigma_{3}$, but this would make $\sigma_1(x), \sigma_2(x) \leq \sigma_{3}(x)$.  Thus $\textit{X}$ a directed set in $L$.  
$$
\phi (x) \wedge \left( \bigvee_{\psi\in \Psi} \sigma_\psi (x) \right) =  \bigvee_{\psi \in \Psi} \left( \phi(x) \wedge \sigma_\psi (x) \right) =  \bigvee_{\psi \in \Psi} \left( (\phi \wedge \sigma_\Psi) (x) \right).
$$ 

Let $\Phi =\left\{\phi \wedge \sigma_d | d \in D \right\}$.  Let $\phi \wedge \sigma_1, \phi \wedge \sigma_2 \in \Phi$ then there is a $\sigma_{3} \in  (\sigma_\psi | \psi \in \Psi )$ such that 
$ \sigma_1, \sigma_2 \leq \sigma_3$.  This would then make 
$\phi \wedge \sigma_1, \phi \wedge \sigma_2 \leq \phi \wedge \sigma_{3}$ which is in $\Phi$.  We then have that $\Phi$ is directed.  

Thus, $\left( \bigvee_{\psi \in \Psi} (\phi \wedge \sigma_\psi)\right) (x) =\bigvee_{\psi \in \Psi} \left( (\phi \wedge \sigma_d) (x) \right)$, and   
$$
\left( \phi \wedge \left(\bigvee_{\psi \in \Psi}\sigma_\psi  \right) \right) (x)  = \left( \bigvee_{\psi \in \Psi} (\phi \wedge \sigma_\psi) \right) (x).
$$ 
Making $a.c.o.(L)$ is a meet-continuous lattice.  
\end{proof}

We will show $a.c.o.(L)$ is algebraic by defining a set of compact elements in $a.c.o.(L)$ that will be the building blocks for our lattice. The proof of this first lemma will be left to the reader. 

\begin{lemma}\label{phi_closure_L}
Let $a, b, v \in L$. Then for $x \in L$ let 

\begin{displaymath}\label{phivab_L}
 \phi^v_{a,b}(x) := \left\{ \begin{array}{cl} \nonumber
 x \vee v \vee b , & \textrm{if } a \, \leq x \lor v  \\ 
 x \vee v, & otherwise. \\
 
 \end{array}\right. 
\end{displaymath}
 Then $\phi^v_{a,b}$ is a closure operator for $L$. 
\end{lemma}

\begin{lemma}\label{phi_alg_L} 
Let $\phi^v_{a,b}$ be defined as in Lemma~\ref{phivab_L}. If $a \in C(L)$ then $\phi^v_{a,b}$ is algebraic. 
\end{lemma}
\begin{proof}
Let $x \in L$.  If $ a \not\leq (x \vee v) $ then $a \not\leq (k \vee v)$ for all $k \leq_c x $.  Thus 
$$
\phi^v_{a,b}(x) = x \vee v = \left(\bigvee_{k \leq_c x} k\right) \vee v =\bigvee_{k \leq_c x } ( k \vee v ) = \bigvee_{k \leq_c x} \phi^v_{a,b}(k).  
$$

If $ a \leq (x \vee v) = \left(\bigvee_{k \leq_c x} k\right) \vee v $ then because $a$ is compact $ a \leq \left(\bigvee_{k \in F} k \right) \vee v$  where $F$ is a finite subset of the compact elements less than $x$.  
Let \\
$k^*~=~\bigvee_{k \in F}~k$.  $k^*$ is a compact element, thus $ a \leq k^* \vee v$ for some $k^* \leq_c x$ and $ \phi^v_{a,b} (k^*) = k^* \vee v \vee b $.  This makes 
$$
\phi^v_{a,b} (x) = x \vee v \vee b = \left(\bigvee_{k \leq_c x} k\right) \vee v \vee b = (k^* \vee v \vee b) \vee \left( \bigvee_{k \leq_c x} k \vee v \right) = \bigvee_{k \leq_c x} \phi^v_{a,b} (k).    
$$
Thus, $\phi^v_{a,b}$ algebraic.   
\end{proof}

\begin{lemma}\label{phi_compact_L}
Let $\phi^v_{a,b}$ be defined as in Lemma~\ref{phivab_L} and let $v, a,b \in C(L)$ then $\phi^v_{a,b}$ is compact in $a.c.o.(L)$. 
\end{lemma}
\begin{proof}
To show that $\phi^v_{a,b}$ is a compact element, we will show that it is join-inaccessible. Let the family $(\sigma_\psi | \psi\in \Psi)$ be a directed set in $a.c.o.(L)$ and where $\phi^v_{a,b}(x)~=~\left(\bigvee_{\psi \in \Psi}\sigma_\psi \right)(x)$ for all $x\in L$.  Since $\phi^v_{a,b}$ is finitary we need only consider the compact elements. Let $c \in C(L)$.\\ 
Case 1 $(a \leq c \vee v)$:  First we will look at when $c = a$,  $ \phi^v_{a,b} (a) =\left(\bigvee_{\psi \in \Psi } \sigma_\psi \right)  (a)=  \bigvee_{\psi \in \Psi} (\sigma_\psi (a))$.  Because $\phi^v_{a,b}(a) = a\vee v \vee b$ is compact in $L$ we can reduce the join in $L$ to get $ \phi^v_{a,b} (a) =  \sigma_a (a)$ for some $\sigma_a \in (\sigma_\psi | \psi \in \Psi )$. Now looking at  $c \geq a$,  
$$
\sigma_a (c)  \leq   \left( \bigvee_{\psi \in \Psi} \sigma_\psi \right) (c)    =   \phi^v_{a,b}(c) = c \vee v \vee b = c \vee a \vee v \vee b = \displaystyle c \vee \sigma_a (a) \leq \sigma_a (c).
$$
 Thus for $c \geq a$, $\phi^v_{a,b} (c) = \sigma_a (c)$.  \\
Case 2 $(c \vee v \not\geq a)$: Similarly to case one we will look at when $c=0$, the least element of $L$.   
$$
0 \vee v = \phi^v_{a,b}(0)=\left( \bigvee_{d \in D} \sigma_d \right) (0) = \bigvee_{d \in D} (\sigma_d (0)) =  \sigma_0 (0)  
$$  
for some $\sigma_0 \in (\sigma_\psi | \psi \in \Psi )$.  Then for $a \not\leq c\vee v$ 
$$
\sigma_0 (c) \leq \phi^v_{a,b}(c)= c \vee v = c \vee \phi^v_{a,b} (0) = c \vee \sigma_0 (0) \leq \sigma_0(c).
$$  
This makes $ \phi^v_{a,b}(c) = \sigma_0 (c)$ for $c \vee v \not\geq a$.  

For any $c \in C(L)$, $\displaystyle \phi^v_{a,b}(c) = \left( \bigvee_{\psi \in D} \sigma_\psi \right) (c) = \left(\sigma_0 \vee \sigma_a \right) (c) = \sigma_b (c)$ for some $\sigma_b \in (\sigma_\psi | \psi \in \Psi )$ given our directed set.  This would make $\phi^v_{a,b}$ join-inaccessible in $a.c.o.(L)$.  Which is to say $\phi^v_{a,b} \in C(a.c.o.(L))$.     
\end{proof}

\begin{lemma}\label{join_compact_L}
 Let $\phi^v_{a,b}$ be defined as in Definition~\ref{phivab_L}. Let $a,b \in C(L)$ but $v$ does not have to be compact.  Then $\phi^v_{a,b}= \bigvee_{k \leq_c v} \phi^k_{a,b}$  
\end{lemma}

\begin{proof}
 First we need to make an observation. If $k \leq_c v$ then \\
$ \left(\bigvee_{k \leq_c v} \phi^k_{a,b}\right)(x) \leq \phi^v_{a,b}(x)\textnormal{ for all } x \in L $.

Since $\phi^v_{a,b}$ and $\phi^k_{a,b}$  are algebraic  showing  $\phi^v_{a,b}(l) \leq \left(\bigvee_{k \leq_c v} \phi^k_{a,b}\right)(l)$ for all $l \in C(L)$ will suffice. 

 Let $l \in C(L)$ such that $ a \leq l \vee v = l \vee \left( \bigvee_{k \leq_c v} k \right)$.
Since $a$ is compact this join can be reduced to a finite join. $ a \leq l \vee \left(\bigvee_{k \in F_a} k \right)$.  
Then for $k_a = \bigvee_{k \in F_a} k$,  $k_a$ is a compact and we have 
$$
\phi^v_{a,b}(l) = l \vee v \vee b  = l \vee k_a \vee b \vee \left(\bigvee_{k \leq_c v} k \right) \leq%
\phi_{a,b}^{k_a} (l) \vee \left(\bigvee_{k \leq_c v} \phi^k_{a,b}(l) \right) \leq \bigvee_{k \leq_c v} \phi^k_{a,b}(l).  
$$
For $a \not\leq l \vee v$, 
$$
\phi^v_{a,b}(l) = l \vee v = l \vee \left( \bigvee_{k \leq_c v} k \right) =\bigvee_{k \leq_c v}\left(\phi^k_{a,b}(l)\right)= \left( \bigvee_{k \leq_c v} \phi^k_{a,b} \right) (l).
$$  
We thus have in either case $\phi^v_{a,b}(l) \leq \bigvee_{k \leq_c v}\phi^k_{a,b}(l)$ for all $l \in C(L)$, $\phi^v_{a,b} =  \bigvee_{k \leq_c v} \phi^k_{a,b}$.  
\end{proof}

\begin{lemma}\label{all_join_compact_L}
 Let $\sigma \in a.c.o.(L)$. Then 
$$
\sigma = \bigvee_{a \in C(L)} \bigvee_{b \leq_c \sigma(a)} \phi^{v}_{a,b}  \ \ \textnormal{ where } v=\sigma(0).
$$
\end{lemma}

\begin{proof}
 First of all, to ease between joins in $L$ and joins in $a.c.o.(L)$ we need to note that $C(L)$ is a directed set, and for any $x \in L$, 
$\displaystyle \left\{l \in C(L) | l  \leq_c x  \right\}$ is also directed.  Thus by Lemma~\ref{dir_join_L} we have 
$$
\left( \bigvee_{a \in C(L)} \bigvee_{b \leq_c \sigma(a)} \phi^{v}_{a,b} \right) (x) = \bigvee_{a \in C(L)} \bigvee_{b \leq_c \sigma(a)} ( \phi^v_{a,b} (x) ).
$$
Let $x \in L$, $a \in C(L)$, and $b \leq_c \sigma(a)$.  If $\phi^v_{a,b}(x) = x \vee v$ and we have $ x \leq \sigma(x)$ and $ v = \sigma(0) \leq \sigma(x)$ 
then $\displaystyle \phi^v_{a,b}(x)= x \vee v = x \vee \sigma(0) \leq \sigma(x)$.  If $a \leq x \vee v$, which is to say $\phi^v_{a,b}(x) = x \vee v \vee b$, then we have $x \vee v  \leq \sigma(x)$ and 
$a, b \leq \sigma(a)$.  These inequalities give us
 $a \leq \sigma(x \vee v) \leq \sigma(\sigma(x)) = \sigma(x)$ which, by closure operator properties, 
yields $\sigma(a) \leq \sigma(x)$.  We then have $a \vee b \leq \sigma(a) \leq \sigma (x)$ so $a \vee b \vee x \vee v =x \vee v \vee b \leq \sigma (x)$.  Thus, if 
$\phi^v_{a,b}(x) = x \vee v$ or $ \phi^v_{a,b}(x) = x \vee v \vee b$ then $\phi^v_{a,b}(x) \leq \sigma (x)$. Since $a$ was chosen arbitrarily this would make 
$$
\left( \bigvee_{a \in C(L)} \bigvee_{b \leq_c \sigma(a)} \phi^{v}_{a,b}\right) (x) = \bigvee_{a \in C(L)} \bigvee_{b \leq_c \sigma(a)} ( \phi^v_{a,b} (x) ) \leq \sigma(x).  
$$

For the other inequality we will use the fact that these closure operators are algebraic and look at compact elements of $L$. Let $ l \in C(L)$ and $k \leq_c \sigma(l)$. Then
$$
k \leq \phi_{l,k}^v (l) \leq \bigvee_{b \leq_c \sigma(l)} \phi_{l,b}^v (x) \leq \left( \bigvee_{a \in C(L)} \bigvee_{b \leq_c \sigma(a)} \phi^{v}_{a,b} \right) (l).
$$
  This is true for all $k \leq_c \sigma(l)$, thus 
$$
\sigma(l) = \bigvee_{k \leq \sigma(l)} k \leq \left( \bigvee_{a \in C(L)} \bigvee_{b \leq_c \sigma(a)} \phi^{v}_{a,b} \right) (l).  
$$  
We have now shown both inequalities and thus $\displaystyle \sigma = \bigvee_{a \in C(L)} \bigvee_{b \leq_c \sigma(a)} \phi^{v}_{a,b}$.   
\end{proof}

\begin{theorem}
The lattice $a.c.o.(L)$ is algebraic. 
\end{theorem}
\begin{proof} 
 from Lemmas \ref{join_compact_L} and \ref{all_join_compact_L} we see for each $\sigma \in a.c.o.(L)$, $\sigma$ can be written as the join of elements that are the join of compacts. Thus $a.c.o.(L)$ is algebraic. 

\end{proof}



\end{document}